\documentclass[reqno]{amsart}
\usepackage[utf8]{inputenc}
\usepackage{amsmath}
\usepackage{amsfonts}
\usepackage{amssymb}
\usepackage{amsthm}
\usepackage{mathrsfs}
\usepackage[a4paper, margin=1.5in]{geometry}

\newtheorem{theorem}{Theorem}
\newtheorem{lemma}{Lemma}

\theoremstyle{definition}
\newtheorem{definition}{Definition}

\newtheorem{remark}{Remark}

\newcommand{\dnint}[1]{\left\|{#1}\right\|}
\renewcommand{\mod}{\;\mathrm{mod}\;}

\title{A Note on Diophantine Approximation with Restricted Denominators}
\author{Chance Sanford}
\address{Madison, Wisconsin, USA}
\email{sanfordchance@gmail.com}

\date{January 2026}

\begin{document}

\begin{abstract}
    In this note we explore rational approximations to irrational numbers whose denominators are restricted to a subset of the natural numbers.  We introduce a specific notion of density for such subsets and use it to establish a restricted analog of a theorem of Dirichlet's.
\end{abstract}

\maketitle
\section{Introduction}
Let us begin by recalling an important corollary of Dirichlet's approximation theorem, found for instance in \cite{schmidt1980}.
\begin{theorem}[Dirichlet]\label{thm:dirichlet}
For any irrational number $\xi$ there are infinitely many coprime pairs $(p,q) \in \mathbb{Z} \times \mathbb{N}$ such that
\begin{equation*}
    \left\vert{\xi- \frac{p}{q}}\right\vert < \frac{1}{q^{2}}.
\end{equation*}
\end{theorem}

With Dirichlet's result in mind, we may ask the following question.   What happens when the denominators of the approximating rationals are restricted to some subset of the natural numbers? One would expect that if the restricting subset is too sparse, then it would be impossible to obtain as strong of a bound as $q^{-2}$ given in Dirichlet's theorem.

A number of authors have considered this question over the years.  Generally, research has focused on restrictions to arithmetically significant subsets of the natural numbers.  For instance, Vinogradov \cite{vinogradov1947} initiated the study of Diophantine approximation with denominators restricted to the set of primes, which we denote by $\mathcal{P}$. Currently, the best result in this direction is the following theorem, due to Matom{\"a}ki \cite{matomaki2009}.
\begin{theorem}[Matom{\"a}ki]\label{thm:matomaki}
For any irrational number $\xi$ and real number $\varepsilon > 0$, there are infinitely many pairs $(a,p) \in \mathbb{Z} \times \mathcal{P}$ such that
\begin{equation*}
    \left\vert{\xi- \frac{a}{p}}\right\vert < \frac{1}{p^{4/3-\varepsilon}}.
\end{equation*}
\end{theorem}

The proof Matom{\"a}ki's theorem, as well as similar results, fall squarely in the domain of analytic number theory.  Consequently, they are of quite a different character than that of the proofs of the classical approximation results, like Dirichlet's theorem given above.

In this paper, we take a more combinatorial approach, introducing a certain notion of density for a subset of $\mathbb{N}$. We show that if $\mathcal{A} \subseteq \mathbb{N}$ is sufficiently dense, then there exists an analog of Dirichlet's theorem where the denominators of the approximating rationals are restricted to $\mathcal{A}$.

Before moving on we need to define some notation.  First, let $\mathbb{N} := \{1,2,3,\dots\}$ denote the set of natural numbers.  Next, we will write $f \ll g$ if, for two functions $f$ and $g$, there exists a constant $C > 0$, such that $\vert{f(x)}\vert \leq C \vert{g(x)}\vert$ for all $x$ in their common domain.  If $f \ll g$ and $g \ll f$, then we write $f \asymp g$.  Lastly, let $\dnint{x}$ denote the distance of $x$ to the nearest integer.

\section{Main results}
We begin this section by defining our notion of \textit{Diophantine density} for a subset of $\mathbb{N}$.  Afterwards, we present and prove the two principle results of this note.
\begin{definition}
    Let $\delta \in [0,1]$. A set $\mathcal{A} \subseteq \mathbb{N}$ is said to have a \textit{Diophantine density} of $\delta$, if there exists a constant $C > 0$, depending only on $\mathcal{A}$, such that for every coprime pair $(p,q) \in \mathbb{Z} \times \mathbb{N}$ with $q$ sufficiently large,
    \begin{equation*}
    \min_{v\in\mathcal{A},v\leq q}\dnint{\frac{v p}{q}} \leq \frac{C}{q^{\delta}}.
\end{equation*}
\end{definition}

\begin{remark}
    Here we would like to point out that, if a set $\mathcal{A} \subseteq \mathbb{N}$ has a Diophantine density $\delta > 0$, then it must contain an infinite number of elements. To prove this, assume that $\mathcal{A}$ is finite with a density greater than $0$, and consider the sequence of convergents $(p_{n}/q_{n})_{n\geq 0}$ of an irrational number $\xi$.  By definition, for large enough $n$ we must have
    \begin{equation*}
    \min_{v\in\mathcal{A},v\leq q_{n}}\dnint{\frac{v p_{n}}{q_{n}}} \leq \frac{C}{q_{n}^{\delta}}.
\end{equation*}
But since $\mathcal{A}$ is finite and the sequence of convergents is infinite, by Dirichlet's \textit{schubfachprinzip}, there must be a single $v \in \mathcal{A}$ such that 
    \begin{equation*}
    \dnint{\frac{v p_{n}}{q_{n}}} \leq \frac{C}{q_{n}^{\delta}},
\end{equation*}
for infinitely many $n$.  Let us denote the subsequence of $n$'s for which this is true by $(n_{k})_{k\geq 0}$. Taking the limit as $k\to\infty$ and keeping in mind that $\delta > 0$, we find that
\begin{equation*}
    \dnint{v \xi} = \lim_{k\to\infty}\dnint{\frac{v p_{n_{k}}}{q_{n_{k}}}} \leq \lim_{k\to\infty}\frac{C}{q_{n_{k}}^{\delta}} = 0.
\end{equation*}
But this implies that $\xi$ is rational, which is a contradiction.  As a result, we conclude that $\mathcal{A}$ cannot be finite.
\end{remark}

Before presenting our first theorem, we recall the following lemma which states that the convergents of an irrational number provide an explicit sequence of rational numbers which satisfy the inequality in Theorem \ref{thm:dirichlet}.  We direct the reader to Schmidt's monograph \cite{schmidt1980} for a proof of this result.

\begin{lemma}\label{lem:cnvgt_diff}
    Let $p_{n}/q_{n}$ be the $n$-th convergent of $\xi$.  Then
    \begin{equation*}
        \left\vert{\xi - \frac{p_{n}}{q_{n}}}\right\vert < \frac{1}{q^2_{n}}.
    \end{equation*}
\end{lemma}

Now we are able to state our first theorem, which is a restricted analog of Dirichlet's theorem.

\begin{theorem}\label{thm:main}
    Suppose that $\mathcal{A} \subseteq \mathbb{N}$ has a Diophantine density of $\delta > 0$.  For any irrational number $\xi$ and $\varepsilon > 0$, there are infinitely many coprime pairs $(u,v) \in \mathbb{Z} \times \mathcal{A}$ such that
    \begin{equation*}
    \left\vert{\xi - \frac{u}{v}}\right\vert < \frac{1}{v^{1+\delta - \varepsilon}}.
    \end{equation*}
\end{theorem}
\begin{proof}
Let $p_{n}/q_{n}$ be the $n$-th convergent of $\xi$.  By Lemma \ref{lem:cnvgt_diff} we know that
\begin{equation*}
    \left\vert{ \xi - \frac{p_{n}}{q_{n}}}\right\vert < \frac{1}{q_{n}^{2}}.
\end{equation*}
The fact that $\mathcal{A}$ has a Diophantine density of $\delta$, implies that for large enough $n$, we can find a $v_{n} \in \mathcal{A}$ such that
\begin{equation*}
    \left\vert{\frac{p_{n}}{q_{n}} - \frac{N_{n}}{v_{n}}}\right\vert < \frac{C}{q_{n}^{\delta}v_{n}}.
\end{equation*}
Here, $N_{n}$ is the nearest integer to $v_{n} p_{n}/q_{n}$ and $C > 0$ is a constant depending only on $\mathcal{A}$. Thus, by the triangle inequality and the fact that $v_{n} \leq q_{n}$, we find that
    \begin{align*}
    \left\vert \xi - \frac{N_{n}}{v_{n}} \right\vert &\leq \left\vert{ \xi - \frac{p_{n}}{q_{n}}}\right\vert + \left\vert \frac{p_{n}}{q_{n}} -\frac{N_{n}}{v_{n}} \right\vert \\
    &\leq \frac{1}{q_{n}^{2}} + \frac{C}{q_{n}^{\delta} v_{n}} \\
    &\leq \frac{1}{v_{n}^{2}} + \frac{C}{v_{n}^{1+\delta}}.
\end{align*}
This gives the bound 
\begin{equation*}
    \left\vert \xi - \frac{N_{n}}{v_{n}} \right\vert \leq \frac{C+1}{v_{n}^{1+\delta}}.
\end{equation*}

So far we have shown that for each sufficiently large $n$, there exists a pair $(N_{n},v_{n}) \in \mathbb{Z} \times \mathcal{A}$, not necessarily coprime, which satisfies the inequality in the previous displayed equation.  A priori, it is not clear that the $\delta$-density condition, which allowed us to find such a pair, produces infinitely many distinct coprime pairs $(u,v) \in \mathbb{Z} \times \mathcal{A}$ for a given irrational number $\xi$.  Nevertheless, we will show that this is the case.

Indeed, if we assume that after reducing each pair $(N_{n},v_{n})$ there are only finitely many coprime pairs $(u,v)$, then there must be a single pair for which $N_{n}/v_{n} = u/v$ for infinitely many $n$.  If we denote the subsequence of $n$'s for which this is true by $(n_{k})_{k\geq 0}$, then we find that
\begin{equation*}
    \left\vert \xi -\frac{u}{v} \right\vert = \lim_{k\to\infty}\left\vert \frac{p_{n_{k}}}{q_{n_{k}}} -\frac{N_{n_k}}{v_{n_k}} \right\vert \leq \lim_{k\to\infty}\frac{C}{q_{n_k}^{\delta} v_{n_k}} = 0,
\end{equation*}
since $\delta > 0$.  This is clearly a contradiction, as $\xi$ is irrational.

Therefore, we may conclude that there are infinitely many distinct coprime pairs $(u,v) \in \mathbb{Z} \times \mathcal{A}$, such that
\begin{equation*}
    \left\vert \xi - \frac{u}{v} \right\vert \leq \frac{C+1}{v^{1+\delta}}.
\end{equation*}
Finally, for any $\varepsilon > 0$, when $v > (C+1)^{1/\varepsilon}$ we have 
\begin{equation*}
    \left\vert \xi - \frac{u}{v} \right\vert \leq \frac{C+1}{v^{1+\delta}} < \frac{1}{v^{1+\delta - \varepsilon}}.
\end{equation*}
With that, our proof is complete.
\end{proof}

Our next theorem provides a method for constructing sets of a given density.  To state it, we need the following definition for the counting function of a subset of the natural numbers.

\begin{definition}
    Let $\mathcal{A} \subseteq \mathbb{N}$ and for $n \geq 1$ let
    \begin{equation*}
        \mathcal{A}(n) := \vert \{a \in \mathcal{A} \, : \, a \leq n\} \vert
    \end{equation*}
    denote the number of elements of $\mathcal{A}$ less than or equal to $n$.
\end{definition}

\begin{theorem}\label{thm:set_diff_ddense}
Let $\mathcal{A} \subseteq \mathbb{N}$ and $\delta \in [0,1]$. If $\mathcal{A}(n) \ll n^{\delta}$ for sufficiently large $n$, then the set $\mathbb{N} \,\backslash \,\mathcal{A}$ has a Diophantine density of $1-\delta$.
\end{theorem}
\begin{proof}
Fix a coprime $(p,q) \in \mathbb{Z} \times \mathbb{N}$ with $q$ sufficiently large.  Our first task is to analyze the multiset
\begin{equation*}
    \Psi(p,q):= \left\{\dnint{\frac{k p}{q}}\,:\,  1 \leq k \leq q \right\}.
\end{equation*}
Once we have done that, it will be straightforward to bound the minimum element in the submultiset
\begin{equation*}
    \left\{\dnint{\frac{k p}{q}}\,:\,  1 \leq k \leq q, k \in \mathbb{N}\,\backslash\,\mathcal{A} \right\}.
\end{equation*}
To this end, note that for each $k$ we have
\begin{equation*}
\dnint{\frac{kp}{q}} = \frac{\left\vert{kp-N_{k}q}\right\vert}{q},
\end{equation*}
where $N_{k}$ is the nearest integer to $kp/q$.  Since $\dnint{x} \leq 1/2$ for any real number $x$, we obtain the bound
\begin{equation*}
\vert{k p -  N_{k} q }\vert \leq \frac{q}{2}.
\end{equation*}
Consequently, we find that
\begin{equation*}
\dnint{\frac{kp}{q}} \in \left\{0,\frac{1}{q},\frac{2}{q},\dots,\frac{\lfloor{q/2}\rfloor}{q}\right\}.
\end{equation*}
Next, by definition we have
\begin{equation*}
    \dnint{\frac{p}{q}} = \dnint{\frac{p \mod q}{q}} = \frac{1}{q}\min\left\{p \mod q, q - (p \mod q)\right\}.
\end{equation*}
Thus, for $p_{1},p_{2} \in \mathbb{Z}$, we have
\begin{equation*}
    \dnint{\frac{p_{1}}{q}} = \dnint{\frac{p_{2}}{q}},
\end{equation*}
if and only if $p_{1} \equiv \pm p_{2} \mod q$.

For our purposes, this enables us to determine when
\begin{equation*}
    \dnint{\frac{k_{1} p}{q}} = \dnint{\frac{k_{2} p}{q}},
\end{equation*}
for $1\leq k_{1},k_{2}\leq q$.  Namely, since $p$ is invertible mod $q$, the equality holds if and only if $k_{1} \equiv \pm k_{2} \mod q$.  Or equivalently, since $k_{1},k_{2} \leq q$, if and only if $k_{1} = k_{2}$ or $k_{1} = q-k_{2}$.

From this fact we deduce that each number in the set
\begin{equation*}
\left\{0,\frac{1}{q},\frac{2}{q},\dots,\frac{\lfloor{q/2}\rfloor}{q}\right\}
\end{equation*}
occurs at least once and at most twice in the multiset $\Psi(p,q)$.  In particular, $0$ occurs once for every $q$ and $1/2$ occurs once for even $q$, with all other elements occurring exactly twice.

Now that we understand the structure of the multiset $\Psi(p,q)$, our next goal is to determine what can be said about the minimum remaining element after removing the submultiset of elements which correspond to elements of $\mathcal{A}$ less than or equal to $q$.  To accomplish this, recall that the requirements of the theorem dictate that for large enough $q$ and $c > 0$,
\begin{equation*}
\mathcal{A}(q) \leq c q^{\delta}.
\end{equation*}
Suppose that we remove the smallest $\lfloor{c q^{\delta}}\rfloor$ elements from the multiset $\Psi(p,q)$.  Since each number $k/q$ occurs at least once for $0 \leq k \leq \lfloor{q/2}\rfloor$, the smallest remaining element may be roughly bounded above by ${\lfloor{c q^{\delta}}\rfloor/q}$.

In other words, we find that
\begin{align*}
\min_{v\in\mathbb{N} \,\backslash \,\mathcal{A}, v\leq q}\dnint{\frac{v p}{q}} \leq \frac{\lfloor{c q^{\delta}}\rfloor}{q} \leq \frac{c}{q^{1-\delta}}.
\end{align*}
We have thus demonstrated that there exists a constant $C > 0$, depending only on $\mathbb{N} \,\backslash \,\mathcal{A}$, such  that for any coprime $(p,q)\in\mathbb{Z}\times\mathbb{N}$ with $q$ sufficiently large,
\begin{equation*}
    \min_{v\in\mathbb{N} \,\backslash \,\mathcal{A}, v\leq q}\dnint{\frac{v p}{q}} \leq \frac{C}{q^{1-\delta}}.
\end{equation*}
This completes the proof.
\end{proof}

In the next section we will construct a set whose Diophantine density can be determined explicitly, as well as present the corresponding restricted analog of Dirichlet's theorem.

\section{Approximations avoiding Piatetski-Shapiro sequences}

In this section, we investigate approximations whose denominators avoid elements of Piatetski-Shapiro sequences, which are sequences of the form $(\lfloor{n^{c}}\rfloor)_{n\geq 1}$ where $c > 1$.  Often, $c$ is defined to be non-integral, but here we do not impose such a restriction.  These sequences are named after Piatetski-Shapiro \cite{piatetski-shapiro1953}, who showed that for $c \in (1,12/11)$, there are infinitely many primes of the form $\lfloor{n^{c}}\rfloor$.

To prepare for the remainder of this section, let us define the set 
\begin{equation*}
    \mathcal{PS}_{c} := \{\lfloor{n^{c}}\rfloor\,\vert\, n \in \mathbb{N}\},
\end{equation*}
consisting of the elements of a given Piatetski-Shapiro sequence.  The following lemma provides an upper bound for the counting function of $\mathcal{PS}_{c}$.

\begin{lemma}\label{lemma:ps_seq}
    The counting function for a Piatetski-Shapiro set satisfies the bound
    \begin{equation*}
    \mathcal{PS}_{c}(n) \ll n^{1/c}.
\end{equation*}
\end{lemma}
\begin{proof}
Our goal is to show that
\begin{equation*}
    \mathcal{PS}_{c}(n) = \begin{cases}
        \lfloor{n^{1/c}}\rfloor &\text{if } c \in \mathbb{N}, \\
        \lfloor{(n+1)^{1/c}}\rfloor &\text{if } c \not\in \mathbb{N}, (n+1)^{1/c} \not \in \mathbb{N}, \\
        (n+1)^{1/c} - 1 &\text{if } c \not\in \mathbb{N}, (n+1)^{1/c} \in \mathbb{N},
    \end{cases}
\end{equation*}
which then implies the desired inequality $\mathcal{PS}_{c}(n) \ll n^{1/c}$.  To accomplish this we will count the number of positive integers $k$ such that $\lfloor{k^{c}}\rfloor \leq n$.

Let us first assume that $c \in \mathbb{N}$.  Then $\lfloor{k^{c}}\rfloor = k^{c}$ and the inequality $k^{c} \leq n$ implies that $k \leq  n^{1/c}$.  There are $\lfloor{n^{1/c}}\rfloor$ of such $k$, which completes the first case.  

If $c \notin \mathbb{N}$, then $\lfloor{k^{c}}\rfloor \leq n$ implies that $k^{c} < n + 1$ and consequently $k < (n + 1)^{1/c}.$  In this case we have two potential outcomes.  If $(n + 1)^{1/c}$ is not an integer, then there are $\lfloor{(n + 1)^{1/c}}\rfloor$ such integers $k$.  On the other hand, if $(n + 1)^{1/c}$ happens to be an integer, then there are $(n + 1)^{1/c} - 1$ integers $k$ less than $(n + 1)^{1/c}$.  With the value of the counting function $\mathcal{PS}_{c}(n)$ established, our proof is now complete.
\end{proof}
\begin{theorem}\label{thm:piatetski-shapiro-dirichlet}
Let $\xi$ be an irrational number and let $c > 1$ be a real number.  For any $\varepsilon > 0$, there are infinitely many coprime pairs $(p,q) \in \mathbb{Z} \times (\mathbb{N}\,\backslash \, \mathcal{PS}_{c})$, such that
     \begin{equation*}
    \left\vert{\xi- \frac{p}{q}}\right\vert < \frac{1}{q^{2-1/c-\varepsilon}}.
\end{equation*} 
\end{theorem}
\begin{proof}
    From Lemma \ref{lemma:ps_seq} we know that $\mathcal{PS}_{c}(n) \ll n^{1/c}$.  Consequently, Theorem \ref{thm:set_diff_ddense} tells us that $\mathbb{N}\,\backslash\,\mathcal{PS}_{c}$ has a density of $1 - 1/c$.  Finally, an application of Theorem \ref{thm:main} completes the proof, keeping in mind that we have required that $c > 1$ so that $1/c < 1$.
\end{proof}
As an example, setting $c = 2$ in Theorem \ref{thm:piatetski-shapiro-dirichlet} shows that for any irrational number $\xi$ and any $\varepsilon > 0$, there are infinitely many coprime pairs $(p,q) \in \mathbb{Z} \times \mathbb{N}$ where $q$ is not a perfect square, such that
\begin{equation*}
    \left\vert{\xi- \frac{p}{q}}\right\vert < \frac{1}{q^{3/2-\varepsilon}}.
\end{equation*}

\section{Closing remarks}
In this note we have demonstrated that if a subset of $\mathbb{N}$ satisfies a suitable density condition, then one can obtain a restricted analog of a theorem of Dirichlet, where the denominators of approximating rationals are members of the given subset.  Moreover, we showed how to construct such subsets by taking complements of subsets of $\mathbb{N}$ whose counting function satisfies a certain growth condition.

In the introduction, we used Matom{\"a}ki's theorem as the paradigmatic example of restricted Diophantine approximation.  Unfortunately, our methods are too coarse to prove similar results.  To see this, let $\mathcal{C}_{1}$ denote the set of composite numbers along with $1$.  The counting function for $\mathcal{C}_{1}$ is given by
\begin{equation*}
    \mathcal{C}_{1}(n) = n - \pi(n), 
\end{equation*}
where $\pi(n)$ is the prime counting function.  It is known that $\pi(n)$ possesses the bounds
\begin{equation*}
    \frac{n}{\log(n)} < \pi(n) < C\frac{n}{\log(n)}
\end{equation*}
for $n > 17$ and $C = 30\log(113)/113$ (see \cite{barkley1962}).  Therefore, for sufficiently large $n$, the counting function for $\mathcal{C}_{1}$ is bounded by
\begin{equation*}
    n - C\frac{n}{\log(n)} <  \mathcal{C}_{1}(n) < n - \frac{n}{\log(n)},
\end{equation*}
showing that $\mathcal{C}_{1}(n) \asymp n$.

Consequently, by Theorem \ref{thm:set_diff_ddense}, the set of primes $\mathcal{P} = \mathbb{N} \,\backslash\, \mathcal{C}_{1}$ has a Diophantine density of $0$, making it unsuitable for use in Theorem \ref{thm:main}.
\bibliographystyle{plain} 
\bibliography{refs} 
\end{document}